\documentclass[11pt,a4paper,reqno]{amsart}
\usepackage[usenames]{color}
\usepackage{amsmath, amsthm, amscd, amsfonts, amssymb}

\newtheorem{theorem}{Theorem}[section]
\newtheorem{lemma}[theorem]{Lemma}

\newtheorem{corollary}[theorem]{Corollary}
\theoremstyle{definition}
\newtheorem{definition}[theorem]{Definition}

\newtheorem{remark}{Remark}[section]
\numberwithin{equation}{section}

\begin{document}
\noindent {\footnotesize\tiny}\\[1.00in]

\textcolor[rgb]{0.00,0.00,1.00}{}
\title[]{On the Iterations of a  Sequence of Strongly Quasi-nonexpansive Mappings with Applications}
\maketitle
\begin{center}
	{\sf Hadi Khat\texttt{}ibzadeh\footnote{E-mail: $^{1}$hkhatibzadeh@znu.ac.ir, $^{2}$mohebbi@znu.ac.ir.} and Vahid Mohebbi$^2$}\\
	{\footnotesize{\it $^{1,2}$ Department of Mathematics, University
			of Zanjan, P. O. Box 45195-313, Zanjan, Iran.}}
\end{center}
\begin{abstract}
In this paper, we study $\Delta$- convergence of iterations for a sequence of strongly quasi-nonexpansive mappings as well as the strong convergence of the Halpern type
regularization of them in Hadamard spaces. Then, we give some their applications in iterative methods, convex and pseudo-convex minimization(proximal point algorithm), fixed point theory and equilibrium problems. The results extend several new results in the literature (for example \cite{bac1, bac-rei, Cholamjiak, fo, is-2, km-1, kr-1, kr-4, llm, qo, tzh, tx, xu}) and some of them seem new even in Hilbert spaces.
\end{abstract}

\section{\bf Introduction and Preliminaries}

Let $(X,d)$ be a metric space. A geodesic from $x$ to $y$ is a map
$\gamma$ from the closed interval $[0,d(x,y)] \subset \mathbb{R}$
to $X$ such that $\gamma(0) = x,\ \gamma(d(x,y)) = y$ and
$d(\gamma(t), \gamma(t')) = |t - t'| $ for all $t, t'\in [0,d(x,y)]$.
The space $(X, d)$ is said to be a
geodesic space if every two points of X are joined by a geodesic. The metric segment $[x,y]$ contains the images of all geodesics, which connect $x$ to $y$. $X$ is called unique geodesic iff $[x,y]$ contains only one geodesic.\\
Let $X$ be a unique geodesic metric space. For each $x, y\in X$ and for each $t\in [0, 1]$, there exists a unique point $z\in [x, y]$ such that
$d(x, z) = td(x, y)$ and $d(y, z) = (1 - t)d(x, y)$. We will use the notation $(1 - t)x \oplus ty$ for the unique point $z$ satisfying in the above statement.

In a unique geodesic metric space $X$, a set $A\subset X$ is called convex iff for each $x,y\in A$, $[x,y]\subset A$.
A unique geodesic metric space $X$ is called CAT(0) space if for all $x, y, z\in X$ and for each $t\in [0, 1]$, we have the following inequality
$$d^2((1 - t)x \oplus ty, z)\leq (1 - t)d^2(x, z) + td^2(y, z)- t(1 - t)d^2(x, y).$$
A complete CAT(0) space is called a Hadamard space.

Berg and Nikolaev in \cite{nik2, nik} have introduced the concept of quasi-linearization
along these lines (see also \cite{aka}). Let us formally denote a pair $(a, b) \in X \times  X$ by $\overset {\rightarrow}{ab}$ and call it a vector. Then quasi-linearization is defined as a map $\langle\cdot,\cdot\rangle : (X \times X)\times(X \times X)\rightarrow \mathbb{R}$ defined by
$$\langle\overset {\rightarrow}{ab},\overset {\rightarrow}{cd}\rangle
=\frac{1}{2}\{d^2(a, d) + d^2(b, c) - d^2(a, c) - d^2(b, d)\}  \   \   \    \ (a, b, c, d \in X).$$
It is easily seen that
$\langle\overset {\rightarrow}{ab},\overset {\rightarrow}{ab}\rangle=d^2(a, b)$,
$\langle\overset {\rightarrow}{ab},\overset {\rightarrow}{cd}\rangle
=
\langle\overset {\rightarrow}{cd},\overset {\rightarrow}{ab}\rangle$,
$\langle\overset {\rightarrow}{ab},\overset {\rightarrow}{cd}\rangle
= -\langle\overset {\rightarrow}{ba},\overset {\rightarrow}{cd}\rangle$
and $\langle\overset {\rightarrow}{ax},\overset {\rightarrow}{cd}\rangle
+
\langle\overset {\rightarrow}{xb},\overset {\rightarrow}{cd}\rangle
=
\langle\overset {\rightarrow}{ab},\overset {\rightarrow}{cd}\rangle$
for all $a, b, c, d, x \in X$. We say that X satisfies the Cauchy-Schwartz inequality if
$\langle\overset {\rightarrow}{ab},\overset {\rightarrow}{cd}\rangle \leq d(a, b)d(c, d)$
for all $a, b, c, d \in X$. It is known (Corollary 3 of \cite{nik}) that a geodesically connected metric space
is a CAT(0) space if and only if it satisfies the Cauchy-Schwartz inequality.

A kind of convergence introduced by Lim \cite{lim} in order to extend weak convergence in CAT(0) setting. Let $(X,d)$ be a Hadamard space, $\{x_k\}$ be a bounded sequence
in $X$ and $x\in X$. Let $r(x,\{x_k\})=\limsup d(x,x_k)$. The
asymptotic radius of $\{x_k\}$ is given by $r(\{x_k\})= \inf
\{r(x,\{x_k\})| x\in X \}$ and the asymptotic center of
$\{x_k\}$ is the set $A(\{x_k\})=\{x \in X|
r(x,\{x_k\})=r(\{x_k\})\}$. It is known that in a Hadamard space,
$A(\{x_k\})$ consists exactly one point.

\begin{definition}
	A sequence $\{x_k\}$ in a Hadamard space $(X,d)$
	$\triangle$-converges to $x\in X$
	if $A(\{x_{k_n}\})=\{x\}$, for each subsequence $\{x_{k_n}\}$ of $\{x_k\}$.\\
It is well-known that every bounded sequence in a Hadamard space has a $\Delta$-convergent subsequence (see \cite{w-b}).
	We denote $\triangle$-convergence in $X$ by
	$\overset{\triangle}{\longrightarrow}$
	and the metric convergence by $\rightarrow$.
\end{definition}

Let  $C\subseteq X$ be closed and convex. Suppose that $T : C\rightarrow C$ is a mapping and $F(T) := \{x \in C : Tx = x\}$. T is said to be nonexpansive (resp. quasi-nonexpansive)
iff $d(Tx,Ty)\leq d(x,y), \  \forall x, y \in C$ (resp. $F(T)\neq\varnothing$ and $d(Tx, q)\leq d(x,q), \  \forall(x, q) \in C \times F(T))$. We recall the definitions of firmly nonexpansive and quasi firmly nonexpansive mappings.
\begin{definition}
	A mapping $T:C\rightarrow C$ is called firmly nonexpansive iff $$\langle \overset{\rightarrow}{xy},\overrightarrow{TxTy}\rangle \geq d^2(Tx,Ty), \  \forall x, y \in C$$
	$T$ is called quasi firmly nonexpansive if $F(T)\neq\varnothing$ and $$\langle \overset{\rightarrow}{xp},\overrightarrow{Txp}\rangle \geq d^2(Tx,p), \  \forall(x, p) \in C \times F(T).$$
\end{definition}

Our definitions of firmly nonexpansive and quasi firmly nonexpansive are extensions of the definitions in Hilbert spaces but they seem different from the corresponding definitions in the literature (see for example \cite{ariz, nico}). We don't know the relation between two definitions of firmly nonexpansive mappings but for quasi-firmly nonexpansive mappings which are more important in this paper, it is easy to check that the usual definition in the literature implies our definition. Therefore our definition is more general than the old definition.

Recently some authors considered the asymptotic behavior of iterations of a (firmly) nonexpansive mapping in geodesic metric spaces specially in Hadamard spaces (see \cite{ariz, nico}). In the next section, we study the asymptotic behavior of iterations of a sequence of quasi firmly nonexpansive mappings as well as the dynamical behavior of their combination with Halperm iteration. We prove our results for more general class of mappings that are strongly quasi nonexpansive sequence.

Following \cite{kr-2}, we recall that $T : C\rightarrow C$ is strongly nonexpansive (resp. strongly quasi nonexpansive)
iff T is nonexpansive and $d(x_k,Tx_k)-d(y_k,Ty_k)\rightarrow 0$, whenever
$\{x_k\}$ and $\{y_k\}$ are sequences in $C$ such that $d(x_k,y_k)$ is bounded and $d(x_k,y_k)-d(Tx_k,Ty_k) \rightarrow 0$ (resp. T is quasi-nonexpansive and $d(x_k,Tx_k)\rightarrow 0$, whenever $\{x_k\}$
is a bounded sequence in C such that $d(x_k,q)-d(Tx_k,q) \rightarrow 0$, for some $q \in F(T)$).
We also recall strongly nonexpansive and strongly quasi-nonexpansive sequences  that play an essential role in this paper. The sequence $\{T_k\}$ of nonexpansive mappings
is said to be strongly nonexpansive sequence iff $d(x_k,T_k x_k)-d(y_k,T_ky_k) \rightarrow 0$,
whenever $\{x_k\}$ and $\{y_k\}$ are sequences in C such that $d(x_k,y_k)$ is bounded and
$d(x_k,y_k)-d(T_kx_k,T_ky_k) \rightarrow 0$ . The sequence $\{T_k\}$ of quasi-nonexpansive mappings is
said to be strongly quasi-nonexpansive sequence iff $\bigcap_kF(T_k)\neq\varnothing$ and $d(x_k,T_k x_k)\rightarrow 0$,
whenever $\{x_k\}$ is a bounded sequence in C such that $d(x_k,q)-d(T_kx_k,q) \rightarrow 0$, for some $q\in\bigcap_kF(T_k)$. It is clear that a strongly nonexpansive sequence $\{T_k\}$ with
$\bigcap_kF(T_k)\neq\varnothing$ is a strongly quasi-nonexpansive sequence.

Let $T:C\rightarrow C$ be a firmly nonexpansive mapping with $F (T)\neq\varnothing$, where $C$ is a closed convex subset of a real Hilbert space $H$. A well-known result implies that  the orbit of an arbitrary point of $H$ under $T$ is convergent weakly to a fixed point of $T$. This result recently has been improved to Hadamard spaces by Ariza-Ruiz et al. in \cite{ariz} and Nicolae in \cite{nico}. They showed that the sequence $x_k=T^kx$ is $\Delta$-convergent to a fixed point of $T$. To achieve strong convergence we need some regularized methods like Halpern regularization which was first used by Xu \cite{xu} in Hilbert spaces and in Hadamard spaces by \cite{sae}.  In this paper we consider the asymptotic behavior of iterations of a sequence of quasi firmly nonexpansive mappings or more generally strongly quasi nonexpansive mappings as well as their Halpern regularization and prove $\Delta$- convergence and strong convergence of their iterations to a common fixed point of the sequence.
In Section 3 of the paper we consider the applications of our results in iterative methods, convex and pseudo-convex minimization, fixed point theory of quasi-nonexpansive mappings and equilibrium problems of pseudo-monotone bifunctions.

\section{\bf Convergence of a Strongly Quasi-nonexpansive Sequence }

To prove the convergence of the iterations $T^kx$, where $T$ is a firmly nonexpansive mapping, demiclosedness of $T$ is essential. $T:C\rightarrow C$ is called demiclosed iff $d(x_k,Tx_k)\rightarrow0$ and $x_k\overset{\triangle}{\longrightarrow} x$ imply $x\in  F(T)$. Demiclosedness is satisfied for nonexpansive mappings but for quasi firmly nonexpansive mappings and strongly quasi nonexpansive mappings we don't have this essential property even in Hilbert spaces, therefore we must assume it. Since we intend to prove convergence for a sequence of strongly quasi nonexpansive mappings we need the definition of demiclosedness for a sequence of mappings.

A sequence $T_k :C\rightarrow C$ of strongly quasi-nonexpansive mappings with $\bigcap_kF(T_k)\neq\varnothing$ is called demiclosed iff
\begin{equation}\begin{cases}
\text{if} \ \{x_{k_j}\}\subset \{x_k\} \  \text{and} \ \{T_{k_j}\}\subset \{T_k\} \ \text{such that} \\
x_{k_j}\overset{\triangle}{\longrightarrow} p\in C \ \text{and} \ \lim d(x_{k_j},T_{k_j}x_{k_j})=0, \ \text{then} \ p\in \bigcap_kF(T_k)
\end{cases}\label{condition}\end{equation}

In this section, we obtain $\Delta$-  convergence of the sequence $\{x_k\}$ given by
\begin{equation}x_{k+1}=T_kx_k\label{iteration}\end{equation} to an element of $\bigcap_kF(T_k)\neq\varnothing$ as well as the strong convergence of the Halpern type algorithm:
\begin{equation}
x_{k+1}=\alpha_k u\oplus(1-\alpha_k)T_kx_k,\\
\label{min-halpern}\end{equation}
to the element $x^*={\rm Proj}_{\bigcap_kF(T_k)}u$, where $u, x_1\in C$ and the sequence $\{\alpha_k\}\subset(0,1)$ satisfies $\lim \alpha_k=0$ and $\sum_{k=1}^{+\infty}\alpha_k= +\infty$. The recent result extends the results of \cite{kr-2} from Hilbert spaces to Hadamard spaces.

\begin{theorem}\label{TTT}
Suppose that $T_k:C\rightarrow C$ is a sequence of strongly quasi-nonexpansive mappings and $x_0\in C$. We define $x_{k+1}=T_k \cdots T_1x_0$ such that $\{T_k\}$ satisfies \eqref{condition}. Then the sequence $\{x_k\}$, $\Delta$-converges to an element of $\bigcap_kF(T_k)$.
\end{theorem}
\begin{proof}
Take $x^*\in \bigcap_kF(T_k)$, then we have
$d(x_{k+1}, x^*)=d(T_kx_k,x^*)\leq d(x_k,x^*)$.
Therefore $\lim d(x_k, x^*)$ exists for all $x^*\in \bigcap_kF(T_k)$, also $\{x_k\}$ is bounded.
Hence, there are $\{x_{k_n}\}$ of $\{x_k\}$ and $p\in C$ such that  $x_{k_n}\overset{\triangle}{\longrightarrow} p\in C$. On the other hand, since $\{T_k\}$ is a sequence of strongly quasi nonexpansive mappings and $\lim d(x_k, x^*)$ exists for all $x^*\in \bigcap_kF(T_k)$, hence  $\lim d(x_{k_n},T_{k_n}x_{k_n})=0$. Now, \eqref{condition} shows that $p\in \bigcap_kF(T_k)$. In the sequel, Opial lemma (see Lemma 2.1 in \cite{kr-3}) follows the result.
\end{proof}

\begin{remark}
Suppose that $\{S_0, S_1, \dots, S_{r-1}\}$ is  a finite family of quasi nonexpansive mappings  which are demiclosed and  define the sequence $\{T_k\}$ by $T_k=S_{[k]}$ where $[k]=k \ (mod \ r) $. Therefore Theorem \ref{TTT} extends Theorem 4.1 of \cite{ariza-lop} in Hadamard space setting.
\end{remark}

\begin{lemma}\cite{ss1} \label{hss1}
	Let $\lbrace s_{k}\rbrace$ be a sequence of nonnegative real numbers, $\lbrace a_{k}\rbrace$ be a sequence of real numbers in $(0, 1)$ with $\sum _{k=1}^{\infty} a_{k} = +\infty$ and $\lbrace t_{k}\rbrace$ be a sequence of real numbers. Suppose that
	$$s_{k+1} \leq (1-a_{k}) s_{k} + a_{k} t_{k},\ \ \ \ \ \forall k\in \mathbb{N}$$
	If $\limsup t_{k_{n}} \leq 0$ for every subsequence $\{s_{k_{n}}\}$ of $\{s_k\}$ satisfying $\liminf (s_{k_{n}+1}-s_{k_{n}})\geq 0$, then $\lim s_{k} = 0$.
\end{lemma}

\begin{theorem}\label{theo-hal}
Suppose that $T_k :C\rightarrow C$ is a sequence of strongly quasi-nonexpansive mappings such that \eqref{condition} is satisfied, then the sequence $\{x_k\}$ generated by (\ref{min-halpern}) converges strongly to ${\rm Proj}_{\bigcap_kF(T_k)}u$.
\end{theorem}

\begin{proof}
Since $ \bigcap_kF(T_k)$ is closed and convex, therefore we assume that  $ x^*={\rm Proj}_{\bigcap_kF(T_k)}u$. By (\ref{min-halpern}), we have:\\
$d(x^*,x_{k+1})\leq \alpha_k d(x^*,u)+(1-\alpha_k)d(x^*,T_kx_k)
\leq \alpha_k d(x^*,u)+(1-\alpha_k)d(x^*,x_k)\leq {\rm max} \{d(x^*,u),d(x^*,x_k)\}
 \leq \cdots \leq {\rm max} \{d(x^*,u),d(x^*,x_1)\}$.\\
Therefore $\{x_k\}$ is bounded. Now, by (\ref{min-halpern}), we have:
$$d^2(x_{k+1},x^*)\leq (1-\alpha_k)d^2(T_kx_k,x^*)+\alpha_kd^2(u,x^*)-\alpha_k(1-\alpha_k)d^2(u,T_kx_k).$$
Since $T_k$ is quasi-nonexpansive, we have: $d^2(x^*,T_kx_k)\leq d^2(x^*,x_k)$, therefore we have:
\begin{equation}
d^2(x_{k+1},x^*)\leq (1-\alpha_k)d^2(x_k,x^*)+\alpha_kd^2(u,x^*)-\alpha_k(1-\alpha_k)d^2(u,T_kx_k).
\label{eq5}\end{equation}
In the sequel, we show $d(x_{k+1},x^*)\rightarrow 0$. By Lemma \ref{hss1}, it suffices to show that $\limsup (d^2(u,x^*)-(1-\alpha_{k_{n}})d^2(u,T_{k_{n}}x_{k_{n}}))\leq 0$  for every subsequence $\{d^2(x_{k_{n}}, x^*)\}$ of $\{d^2(x_k,x^*)\}$ satisfying  $ \liminf (d^2(x_{k_{n}+1}, x^*) - d^2(x_{k_{n}}, x^*)) \geq 0$.\\
	For this, suppose that $\{d^2(x_{k_{n}}, x^*)\}$ is a subsequence of $\{d^2(x_k,x^*)\}$ such that $\liminf (d^2(x_{k_{n}+1}, x^*) - d^2(x_{k_{n}}, x^*)) \geq 0$. Then\\
$0\leq\liminf(d^2(x^*,x_{{k_n}+1})-d^2(x^*,x_{k_n}))\leq \liminf(\alpha_{k_n}d^2(x^*,u)+(1-\alpha_{k_n})d^2(x^*,T_{k_n}x_{k_n}) \\ -d^2(x^*,x_{k_n}))
=\liminf(\alpha_{k_n}(d^2(x^*,u)-d^2(x^*,T_{k_n}x_{k_n}))+d^2(x^*,T_{k_n}x_{k_n}) -d^2(x^*,x_{k_n}))\\
\leq\limsup\alpha_{k_n}(d^2(x^*,u)-d^2(x^*,T_{k_n}x_{k_n}))+\liminf (d^2(x^*,T_{k_n}x_{k_n}) -d^2(x^*,x_{k_n}))\\
= \liminf (d^2(x^*,T_{k_n}x_{k_n}) -d^2(x^*,x_{k_n}))\leq\limsup (d^2(x^*,T_{k_n}x_{k_n}) -d^2(x^*,x_{k_n}))\leq0$.\\
Therefore, we conclude that $\lim (d^2(x^*,T_{k_n}x_{k_n})-d^2(x^*,x_{k_n}))=0$, hence by the definition, we get $\lim d^2(x_{k_n},T_{k_n}x_{k_n})=0$.\\
On the other hand, there are a subsequence $\{x_{k_{n_i}}\}$ of $\{x_{k_n}\}$ and $p\in C$ such that
$x_{k_{n_i}}\overset{\triangle}{\longrightarrow}p$ and
$$\limsup (d^2(u,x^*)-(1-\alpha_{k_n})d^2(u,T_{k_n}x_{k_n}))=\lim (d^2(u,x^*)-(1-\alpha_{k_{n_i}})d^2(u,T_{k_{n_i}}x_{k_{n_i}}))$$
Since $x_{k_{n_i}}\overset{\triangle}{\longrightarrow}p$ and $\lim d(x_{k_{n_i}},T_{k_{n_i}}x_{k_{n_i}})=0$, by \eqref{condition} we have  $p\in \bigcap_kF(T_k)$. On the other hand, $x^*={\rm Proj}_{\bigcap_kF(T_k)}u$, hence we have:
$$\limsup (d^2(u,x^*)-(1-\alpha_{k_n})d^2(u,T_{k_n}x_{k_n}))\leq d^2(u,x^*)-d^2(u,p)\leq0.$$
Therefore Lemma \ref{hss1} shows that
$d(x_{k+1},x^*)\rightarrow 0,$
i.e. $x_k\rightarrow x^*={\rm Proj}_{\bigcap_kF(T_k)}u.$
\end{proof}

\section{\bf Applications}
In this section, we present some examples of strongly quasi nonexpansive sequences and give some applications of the main results in the previous section in iterative methods, optimizatin, fixed point theory and equilibrium problems.
\subsection{\bf Application to Iterative Methods}

 Consider the following iteration which is called Ishikawa iteration.
\begin{equation}
x_{k+1}=(1-\alpha_k)x_k\oplus\alpha_kT((1-\beta_k)x_k\oplus\beta_kTx_k),
\label{Ishikawa-1}\end{equation}
where $T$ is a quasi-nonexpansive mapping and $\alpha_k, \beta_k\in(0,1)$  are two sequences with suitable assumptions. Define
\begin{equation}
 T_k:=(1-\alpha_k)I\oplus\alpha_kT((1-\beta_k)I\oplus\beta_kT),
\label{ish-seq}\end{equation}
where $I$ is the identity mapping. We will prove that $\{T_k\}$ is a strongly quasi-nonexpansive sequence and it satisfies \eqref{condition}. Then we apply the main results to conclude $\Delta$- convergence of Ishikawa iteration and the strong convergence of the Halpern-Ishikawa iteration.
\begin{lemma}\label{lem-ishikawa-1}
Let $T:C\rightarrow C$ be a quasi-nonexpansive mapping. If $\alpha_k\in (0,1)$ be a sequence such that $\limsup \alpha_k <1$, then the sequence $\{T_k\}$ defined by \eqref{ish-seq} is strongly quasi-nonexpansive.
\end{lemma}
\begin{proof}
Take $\{x_k\}$ in $C$ and $p\in F(T)$. Now, by definition of $T_k$, we have\\
$d^2(T_kx_k, p)\leq(1-\alpha_k)d^2(x_k,p)+\alpha_kd^2(T(1-\beta_k)x_k\oplus\beta_kTx_k), p)-\alpha_k(1-\alpha_k)d^2(x_k,T((1-\beta_k)x_k\oplus\beta_kTx_k))
\leq (1-\alpha_k)d^2(x_k,p)+\alpha_kd^2(x_k,p)
-\frac{1-\alpha_k}{\alpha_k} d^2(x_k,T_kx_k)$.\\
Therefore
$$\frac{1-\alpha_k}{\alpha_k} d^2(x_k,T_kx_k)\leq d^2(x_k,p)-d^2(T_kx_k,p),$$
which shows $T_k$ is strongly quasi-nonexpansive.
\end{proof}
Now, we show the sequence $\{T_k\}$ satisfies \eqref{condition}.
\begin{lemma}\label{lem-ishikawa-2}
Let $T:C\rightarrow C$ be a demiclosed and quasi-nonexpansive mapping. If $\alpha_k, \beta_k\in(0,1)$ are two sequences such that $0<\liminf \alpha_k\leq \limsup \alpha_k <1$ and $\beta_k\rightarrow 0$, then the sequence $\{T_k\}$ defined by \eqref{ish-seq} satisfies \eqref{condition}.
\end{lemma}
\begin{proof}
Let $\{x_k\}$ be an arbitrary sequence such that $x_k\overset{\triangle}{\longrightarrow} p$ and $d(x_k, T_k x_k)\rightarrow 0$. We have to prove $p\in \bigcap_k F(T_k)$. The definition of  $T_k$ together with $d(x_k, T_kx_k)\rightarrow 0$ imply that $\alpha_kd(x_k, T((1-\beta_k)x_k\oplus\beta_kTx_k))\rightarrow 0$, hence we have $d(x_k,T((1-\beta_k)x_k\oplus\beta_kTx_k))\rightarrow 0$. Now, set $y_k=(1-\beta_k)x_k\oplus\beta_kTx_k$. We show that $d(y_k,Ty_k)\rightarrow 0$. On the other hand, since $T$ quasi nonexpansive therefore $F(T)\neq\varnothing$ and hence $d(Tx_k,p)\leq d(x_k,p)$ for all $p\in F(T)$, therefore $\{Tx_k\}$ is bounded. Now, since $x_k\overset{\triangle}{\longrightarrow} p$ and $\beta_k\rightarrow 0$, $y_k\overset{\triangle}{\longrightarrow} p$. Note that $d(x_k,Ty_k)\rightarrow 0$, hence we have
$$d(y_k,Ty_k)\leq d(y_k,x_k)+d(x_k,Ty_k)=\beta_k d(x_k,Tx_k)+d(x_k,Ty_k)\rightarrow 0.$$ Now, $y_k\overset{\triangle}{\longrightarrow} p$ and demiclosedness of $T$ imply $p\in F(T)$, i.e. $p\in \bigcap_k F(T_k)$.
\end{proof}
\begin{remark} With assumptions of Lemma 3.2, if $\beta_k\rightarrow0$, then $F(T)=\cap_{k}F(T_k)$.
The following corollary implies that the generated sequence by \eqref{Ishikawa-1} $\Delta$-converges to an element of $F(T)$.
\end{remark}
\begin{theorem}\label{theo-ishikawa-1}
Let $T:C\rightarrow C$ be a demiclosed and quasi-nonexpansive mapping. If $\alpha_k, \beta_k\in(0,1)$ are two sequences such that $\limsup \alpha_k <1$ and $\beta_k\rightarrow 0$, then the sequence  $x_{k}$ generated by \eqref{Ishikawa-1}  $\Delta$-converges to an element of $F(T)$.
\end{theorem}
\begin{proof}
It is a consequence of Lemma \ref{lem-ishikawa-1}, Lemma \ref{lem-ishikawa-2}, Remark 3.1 and Theorem \ref{TTT}.
\end{proof}
Now, we prove the strong convergence of the generated sequence by \eqref{Ishikawa-1} to an element of $F(T)$.
\begin{theorem}\label{theo-ishikawa-2}
Let $T:C\rightarrow C$ be a demiclosed  and quasi-nonexpansive mapping. If $\alpha_k, \beta_k\in(0,1)$ are two sequences such that $\limsup \alpha_k <1$ and $\beta_k\rightarrow 0$, then the sequence $\{x_k\}$ generated by
$$x_{k+1}=\gamma_k u\oplus(1-\gamma_k)T_k x_k,$$
where $\{T_k\}$ is defined by \eqref{ish-seq}, $u, x_1\in C$ and the sequence $\gamma_k\in(0,1)$ satisfies $\lim \gamma_k=0$ and $\sum_{k=1}^{+\infty}\gamma_k= +\infty$ converges strongly to  ${\rm Proj}_{F(T)}u$.
\end{theorem}
\begin{proof}
$\{T_k\}$ is a strongly quasi-nonexpansive sequence by Lemma \ref{lem-ishikawa-1}. Also Lemma \ref{lem-ishikawa-2} shows that the sequence $T_k$
satisfies \eqref{condition}. Now, Theorem \ref{theo-hal} and Remark 3.1 imply that $\{x_k\}$ converges strongly to  ${\rm Proj}_{F(T)}u$.
\end{proof}
 If we take $\beta_k\equiv0$ in \eqref{Ishikawa-1}, then we gain the Mann iteration, i.e.
\begin{equation}
x_{k+1}=(1-\alpha_k)x_k\oplus\alpha_kTx_k,
\label{mann-1}\end{equation}
\begin{corollary}
  Let $T:C\rightarrow C$ be a demiclosed and quasi-nonexpansive mapping. If $\alpha_k\in(0,1)$ is a sequence such that $\limsup \alpha_k <1$,  then the sequence $\{x_k\}$ generated by \eqref{Ishikawa-1}  $\Delta$-converges to an element of $F(T)$.
\end{corollary}
\begin{proof}
It is a consequence of Theorem \ref{theo-ishikawa-1}.
\end{proof}
\begin{corollary}
  Let $T:C\rightarrow C$ be a  demiclosed  and quasi-nonexpansive mapping. If $\alpha_k\in(0,1)$ is a sequence such that $\limsup \alpha_k <1$, then the sequence $\{x_k\}$ generated by
$$x_{k+1}=\gamma_k u\oplus(1-\gamma_k)((1-\alpha_k)x_k\oplus\alpha_kTx_k),$$
where $u, x_1\in C$ and the sequence $\gamma_k\in(0,1)$ satisfies $\lim \gamma_k=0$ and $\sum_{k=1}^{+\infty}\gamma_k= +\infty$, converges strongly to  ${\rm Proj}_{F(T)}u$.
\end{corollary}
\begin{proof}
A consequence of Theorem \ref{theo-ishikawa-2}.
\end{proof}

\subsection{\bf Applications to Proximal Point Algorithms}
This section contains two subsection. First we apply our main results to proximal point algorithm to approximate a minimizer of a convex or pseudo-convex function and in the second subsection we consider a Lipschitz quasi-nonexpansive mapping to approximate a fixed point of it by the proximal method. In the best of our knowledge some of the results in this section are new even in Hilbert spaces.
\subsubsection{\bf Convex and Pseudo-convex Minimization}

In this subsection, we show some applications of our main results of Theorems \ref{TTT} and \ref{theo-hal} to convex and pseudo-convex minimization.\\
A function $f:X\rightarrow]-\infty,+\infty]$ is called\\ (i) convex
iff
\begin{center}
 $f(t x\oplus(1-t)y)\leq t f(x)+(1-t)f(y),\ \ \forall x,y\in X\ \text{and} \ \forall \  0\leq t\leq1$
\end{center}
(ii) quasi convex iff
\begin{center}
$f(t x\oplus(1-t)y)\leq\max\{ f(x),f(y)\},\ \ \forall x,y\in X\ \text{and} \ \forall \ 0\leq t\leq1$
\end{center}
equivalently, for each $r\in \mathbb{R}$, the sub-level set $L^f_r:=\{x\in X:\ \ f(x)\leq r\}$ is a convex subset of $X$.\\
(iii) $\alpha$-weakly convex for some $\alpha>0$ iff
\begin{center}
$f(t x\oplus(1-t)y)\leq t f(x)+(1-t)f(y)+\alpha t(1-t)d^2(x,y),\ \ \forall x,y\in~X\ \text{and} \ \forall \ 0\leq t\leq1$
\end{center}
(iv)  pseudo-convex  iff\\
$f(y)>f(x)$ implies that there exist $\beta(x,y)>0$  and $0<\delta(x,y)\leq1$  such that $f(y)-f(tx\oplus(1-t)y)\geq t\beta(x,y)$,  $\forall t\in(0,\delta(x,y))$.\\
\begin{definition}
	Let $f:X\rightarrow]-\infty,+\infty]$. The domain of $f$ is defined by $D(f):=\{x\in X:\ f(x)<+\infty\}$. $f$ is proper iff $D(f)\neq\varnothing$.
\end{definition}

\begin{definition}
	A function $f:X\rightarrow]-\infty,+\infty]$ is called ($\Delta$-)lower semicontinuous (shortly, lsc) at $x\in D(f)$ iff $$\liminf_{n\rightarrow \infty}f(y_n)\geq f(x)$$
	for each sequence $y_n\rightarrow x$ ($y_n\overset{\triangle}{\longrightarrow}x$) as $n\rightarrow+\infty$. $f$ is called ($\Delta$-)lower semicontinuous iff it is ($\Delta$-)lower semicontinuous in each point of its domain. It is easy to see that every lower semicontinuous and quasi-convex function is $\Delta$-lower semicontinuous.
\end{definition}

Let $f:X\rightarrow]-\infty,+\infty]$ be a convex, proper and lower semicontinuous (shortly, lsc) function where $X$ is a Hadamard space. The resolvent of $f$ of order $\lambda>0$ is defined at each point $x\in X$ as follows:
$$J_{\lambda}^fx:={\rm Argmin}_{y\in X}\{f(y)+\frac{1}{2\lambda}d^2(y,x)\}$$

Existence and uniqueness of $J_{\lambda}^fx$ for each $x\in X$ and $\lambda>0$ was proved by Jost (see Lemma 2 in \cite{J2}). A similar argument shows the existence and uniqueness of the resolvent for $\alpha$-weakly convex function $f$ when $\lambda<\frac{1}{2\alpha}$. The behavior of iterations the resolvent on an arbitrary point of a Hadamard space (named the proximal point algorithm) was proved by Bacak \cite{bac1}, which extends the corresponding result proved by Martinet \cite{ma} in Hilbert spaces (see also Rockafellar \cite{ro}). In this section we conclude $\Delta$-  convergence of the proximal point algorithm as a consequence of Theorem \ref{TTT}. Also we prove the strong convergence of Halpern type proximal point algorithm as a consequence of Theorem \ref{theo-hal}. The last result extends a result by Cholamjiak \cite{Cholamjiak}. First we prove the sequence $J_{\lambda_k}^f$  of mappings satisfies the conditions of Theorems \ref{TTT} and \ref{theo-hal}.

\begin{lemma}\label{lem-qfn}
  Let $f:X\rightarrow]-\infty,+\infty]$ be a quasi-convex, $\alpha$-weakly convex, proper and lsc function. If ${\rm Argmin} f\neq\varnothing$, then $J_{\lambda}^f$ is a quasi firmly nonexpansive mapping for each $\lambda <\frac{1}{2\alpha}$.
\end{lemma}
\begin{proof}
  Taking $\tilde{x}\in {\rm Argmin} f$, $y=t\tilde{x}\oplus(1-t)J_{\lambda}^fx$ and using quasi-convexity of $f$, we get
	 $$f(J_{\lambda}^fx)+\frac{1}{2\lambda} d^2(J_{\lambda}^fx,x)\leq f(J_{\lambda}^fx)+\frac{1}{2\lambda}\{td^2(\tilde{x},x)+(1-t)d^2(J_{\lambda}^fx,x)-t(1-t)d^2(J_{\lambda}^fx,\tilde{x})\}$$
	 By letting $t\rightarrow0^+$, we receive to
	 $$d^2(J_{\lambda}^fx,x)-d^2(x,\tilde{x})+d^2(J_{\lambda}^fx,\tilde{x})\leq0.$$
	 Therefore $\langle\overrightarrow{J_{\lambda}^fx\tilde{x}},\overrightarrow{J_{\lambda}^fxx}\rangle\leq0$ which implies that $d^2(J_{\lambda}^fx, \tilde{x})\leq\langle\overrightarrow{J_{\lambda}^fx\tilde{x}},\overrightarrow{x\tilde{x}}\rangle$. Thus $J_{\lambda}^f$ is quasi firmly nonexpansive mapping.
\end{proof}
\begin{lemma}\label{arg-fix}
Let $f:X\rightarrow]-\infty,+\infty]$ be a quasi-convex, $\alpha$-weakly convex, proper and lsc function. If $\lambda\leq\frac{1}{2\alpha}$, then ${\rm Argmin} f \subseteq F(J_{\lambda}^f)$, moreover, if $f$ is pseudo-convex, then ${\rm Argmin} f = F(J_{\lambda}^f)$.
\end{lemma}
\begin{proof}
It is clear that ${\rm Argmin} f \subseteq F(J_{\lambda}^f)$. Now, we assume that $f$ is pseudo convex and $x\in F(J_{\lambda}^f)$, but $x\not\in {\rm Argmin} f$. Therefore there is $z\in X$ such that $f(x)>f(z)$, hence there are $\beta(x,z)>0$ and $0<\delta(x,z)\leq1$ such that
$$f(tz\oplus(1-t)x)+t\beta(x,z)<f(x),\ \ \ \ \ \ \forall t\in (0,\delta(x,z))$$
On the other hand, since $x\in F(J_{\lambda}^f)$ we have
$$f(x)\leq f(tz\oplus(1-t)x)+\frac{1}{2\lambda}d^2(tz\oplus(1-t)x,x)$$
Therefore we obtain
$$t\beta(x,z)<\frac{1}{2\lambda}d^2(tz\oplus(1-t)x,x)=\frac{t^2}{2\lambda}d^2(x,z)$$
hence $\beta(x,z)<\frac{t}{2\lambda}d^2(x,z)$, thus when $t\rightarrow 0$, we gain contradiction.
\end{proof}

\begin{remark}
In Lemma 3.2 if we define $f:\mathbb{R}\rightarrow]-\infty,+\infty]$, by $f(x)=3x^4-16x^3+24x^2$, then ${\rm Argmin }f\subset F(J_{\lambda}^f)$.
\end{remark}

\begin{remark}
In the previous lemma, if $\mu<\lambda$, then $F(J_{\lambda}^f)\subseteq F(J_{\mu}^f)$. By definition of resolvent
$$f(J_{\mu}^f x)+\frac{1}{2\mu}d^2(J_{\mu}^f x,x)\leq f(J_{\lambda}^f x)+\frac{1}{2\mu}d^2(J_{\lambda}^f x,x)$$
and
$$f(J_{\lambda}^f x)+\frac{1}{2\lambda}d^2(J_{\lambda}^f x,x)\leq f(J_{\mu}^f x)+\frac{1}{2\lambda}d^2(J_{\mu}^f x,x)$$
By summing the above two inequalities, we conclude that
$$(\frac{1}{2\mu}-\frac{1}{2\lambda})d^2(J_{\mu}^f x,x)\leq (\frac{1}{2\mu}-\frac{1}{2\lambda})d^2(J_{\lambda}^f x,x)$$
which implies that $F(J_{\lambda}^f)\subseteq F(J_{\mu}^f)$.
\end{remark}

\begin{lemma}\label{lem-convex}
Let $f:X\rightarrow]-\infty,+\infty]$ be a convex, proper and lsc function. If $\liminf \lambda_k >0$, then $J_{\lambda_k}^f$ satisfies \eqref{condition}.
\end{lemma}
\begin{proof}
Let $\{x_k\}$ be an arbitrary sequence such that $x_k\overset{\triangle}{\longrightarrow} p$ and $d(x_k, J_{\lambda_k}^f x_k)\rightarrow 0$. We want to prove $p\in \bigcap_k F(J_{\lambda_k}^f)$. Note that
$$f(J_{\lambda_k}^f x_k)+\frac{1}{2\lambda_k} d^2(J_{\lambda_k}^f x_k,x_k)\leq f(y)+\frac{1}{2\lambda_k}d^2(y,x_k)$$
Set $y=tJ_{\lambda_k}^f x_k\oplus(1-t)z$, where $t\in (0,1)$ and $z\in X$, then we have
$$f(J_{\lambda_k}^f x_k)+\frac{1}{2\lambda_k} d^2(J_{\lambda_k}^f x_k,x_k)$$
$$\leq tf(J_{\lambda_k}^f x_k)+(1-t)f(z)+\frac{1}{2\lambda_k}
(td^2(J_{\lambda_k}^f x_k,x_k)+(1-t)d^2(z,x_k)-t(1-t)d^2(z,J_{\lambda_k}^f x_k))$$
Therefore
$$f(J_{\lambda_k}^f x_k)-f(z)\leq \frac{1}{2\lambda_k}(d^2(z,x_k)-d^2(J_{\lambda_k}^f x_k,x_k)-td^2(z,J_{\lambda_k}^f x_k))$$
By taking $t\rightarrow1^-$, we can conclude that
$$f(J_{\lambda_k}^f x_k)-f(z)\leq \frac{1}{\lambda_k}\langle\overrightarrow{zJ_{\lambda_k}^f x_k}, \overrightarrow{J_{\lambda_k}^f x_kx_k}\rangle$$
Now, by Cauchy-Schwartz inequality, we have
$$f(J_{\lambda_k}^f x_k)-f(z)\leq \frac{1}{\lambda_k}d(z,J_{\lambda_k}^f x_k)d(J_{\lambda_k}^f x_k,x_k)$$
Since $\liminf \lambda_k >0$, taking liminf and $\triangle$-lower semicontinuity of $f$ shows that $f(p)\leq f(z)$ for all $z\in X$. Hence $p\in {\text{Argmin} f}$ which implies that $p\in \bigcap_k F(J_{\lambda_k}^f)$.
\end{proof}

It is valuable that the following theorem extends the results of \cite{Cholamjiak}.
\begin{theorem}
 Let $f:X\rightarrow]-\infty,+\infty]$ be a convex, proper and lsc function. If $\liminf \lambda_k>0$ and ${\rm Argmin} f\neq\varnothing$, then the sequence $\{x_k\}$ generated by
$$x_{k+1}=\alpha_k u\oplus(1-\alpha_k)J_{\lambda_k}^f x_k,$$
where $u, x_1\in C$ and the sequence $\{\alpha_k\}\subset(0,1)$ satisfies $\lim \alpha_k=0$ and $\sum_{k=1}^{+\infty}\alpha_k= +\infty$, converges strongly to  ${\rm Proj}_{{\rm Argmin} f}u$.
\end{theorem}

\begin{proof}
Lemma \ref{lem-convex} implies that $J_{\lambda_k}^f$ satisfies \eqref{condition}. Also by Lemma \ref{lem-qfn} $J_{\lambda_k}^f$ is a quasi firmly nonexpansive sequence and therefore strongly quasi-nonexpansive sequence. Now, Theorem \ref{theo-hal} and Lemma 3.8 imply that $\{x_k\}$ converges strongly to  ${\rm Proj}_{{\rm Argmin} f}u$.
\end{proof}

\begin{lemma}\label{lem-closed}
Suppose that $f:X\rightarrow]-\infty,+\infty]$ is proper, lsc and pseudo-convex function and $\liminf \lambda_k>0$, then the sequence $J_{\lambda_k}^f$ is closed. i.e. if $x_k\rightarrow p$ and $d(J_{\lambda_k}^fx_k,x_k)\rightarrow0$ as $k\rightarrow+\infty$, then $p\in F(J^f_{\lambda_k})$ for each $k\geq1$.
\end{lemma}

\begin{proof}
Suppose that $d(J_{\lambda_k}^fx_k,x_k)\rightarrow0$ and $x_k\rightarrow p$ as $k\rightarrow+\infty$. Then by the definition of $J_{\lambda_k}^fx_k$, we get
$$f(J_{\lambda_k}^f x_k)+\frac{1}{2\lambda_k} d^2(J_{\lambda_k}^f x_k,x_k)\leq f(y)+\frac{1}{2\lambda_k}d^2(y,x_k),\ \ \ \forall y\in X$$
By letting $k\rightarrow+\infty$ and using lower semicontinuity of $f$, we get:
$f(p)\leq f(y)+\frac{1}{2\lambda}d^2(y,p)$ where $\liminf \lambda_k>\lambda>0$.
Now, set $y=J^f_{\lambda}p$, we get
$f(p)\leq f(J^f_{\lambda}p)+\frac{1}{2\lambda}d^2(p,J^f_{\lambda}p)$.
By the definition of $J^f_{\lambda}p$ we get:
$f(p)= f(J^f_{\lambda}p)+\frac{1}{2\lambda}d^2(p,J^f_{\lambda}p)$.
If $p\neq J^f_{\lambda}p$, then $f(J^f_{\lambda}p)<f(p)$, then there exists $\beta(J^f_{\lambda}p,p)>0$ and $0<\delta(J^f_{\lambda}p, p)\leq1$ such that
$f(tJ^f_{\lambda}p\oplus(1-t)p)+t\beta(J^f_{\lambda}p,p)<f(p)$ for all $t\in(0,\delta(J^f_{\lambda}p, p))$.
On the other hand by the definition of $J^f_{\lambda}p$, we have
$f(J^f_{\lambda}p)+\frac{1}{2\lambda}d^2(J^f_{\lambda}p,p)\leq f(tJ^f_{\lambda}p\oplus(1-t)p)+\frac{t^2}{2\lambda}d^2(p,J^f_{\lambda}p)$.
Therefore $f(p)-\frac{t^2}{2\lambda}d^2(p,J^f_{\lambda}p)+t\beta(J^f_{\lambda}p,p)<f(p)$
by letting $t\rightarrow0$ we get
$\beta(J^f_{\lambda}p,p)\leq0$ which is a contradiction.
Hence $p\in F(J^f_{\lambda})$.
\end{proof}

\begin{theorem}
 Let $f:X\rightarrow]-\infty,+\infty]$ be  an $\alpha$-weakly convex, pseudo-convex, proper and lsc function where $X$ is a locally compact Hadamard space. Suppose that $\liminf \lambda_k>0$ and ${\rm Argmin} f\neq\varnothing$. Then the sequence $\{x_k\}$ generated by $x_{k+1}=J_{\lambda_k}^fx_k$ (proximal point algorithm)
 converges to an element of ${\rm  Argmin}f$.
\end{theorem}

\begin{proof} A consequence of Theorem \ref{TTT}, Lemmas 3.7, 3.8 and \ref{lem-closed} and Proposition 4.4 of \cite{ak}.
\end{proof}

\subsubsection{\bf Fixed Point of a Lipschitz Quasi-nonexpansive Mapping}

In this subsection we apply our main results in Section 2 to approximate a fixed point of a Lipschitz quasi-nonexpansive mapping by the proximal point algorithm. Similar to the previous section we must prove the resolvent of a Lipschitz quasi-nonexpansive mapping satisfies the conditions of Theorems \ref{TTT} and \ref{theo-hal}. First we recall the definition as well as existence and uniqueness of the resolvent. The resolvent operator $J_{\lambda}^T$ for a nonexpansive mapping $T$ has been defined in the literature for Hadamard spaces (see \cite{bac-rei, kr-3}). The definition for a Lipschitz mapping is similar but it exists only for some parameters $\lambda$.

Let $C\subseteq X$ be closed and convex. Suppose that  $T:C\rightarrow C$ is a mapping and $\alpha>1$ such that $d(Tx,Ty)\leq \alpha d(x,y)$. For $\lambda>0$ and $x\in C$, we define
$T^x_{\lambda}:C\rightarrow C$ as
$$T^x_{\lambda}y=\frac{1}{1+\lambda}x\oplus\frac{\lambda}{1+\lambda}Ty.$$
Now, take $y_1, y_2\in C$, then note that\\
$d(T^x_{\lambda}y_1,T^x_{\lambda}y_2)=d(\frac{1}{1+\lambda}x\oplus\frac{\lambda}{1+\lambda}Ty_1, \frac{1}{1+\lambda}x\oplus\frac{\lambda}{1+\lambda}Ty_2)\leq \frac{\lambda}{1+\lambda}d(Ty_1,Ty_2)\leq \frac{\alpha\lambda}{1+\lambda}d(y_1,y_2)$. In the sequel, if $\frac{\alpha\lambda}{1+\lambda}<1$, then $T^x_{\lambda}$ is a contraction, i.e. if $\lambda<\frac{1}{\alpha-1}$ then $T^x_{\lambda}$ has a unique fixed point which we denote it by $J^T_{\lambda}x$ and it is called the resolvent of $T$ of order $\lambda>0$ at $x$. In fact, $J^T_{\lambda}x=F(T^x_{\lambda})$.
It is easy to see that $F(J_{\lambda}^T)=F(T)$.
First, suppose that $J_{\lambda}^Tx=x$ therefore $x=\frac{1}{1+\lambda}x\oplus\frac{\lambda}{1+\lambda}Tx$ which implies that $Tx=x$. Now, suppose $Tx=x$ hence $x=\frac{1}{1+\lambda}x\oplus\frac{\lambda}{1+\lambda}Tx$, therefore $J_{\lambda}^Tx=x$.
\begin{remark}
If $X=H$ a Hilbert space  and $T$ and $I$ are respectively a nonexpansive and identify mappings, then the resolvent of the maximal monotone operator $I-T$ is exactly $J^T_{\lambda}$ which was defined above.
\end{remark}
In the sequel, we will prove the $\Delta$-convergence of generated sequence by \eqref{ppa}.
Now, let $T:C\rightarrow C$ be aa $\alpha$-Lipschitz and quasi-nonexpansive mapping, where $C$ is closed and convex and $\lambda_k<\frac{1}{\alpha-1}$, we define
$J_{\lambda_k}^T:C\rightarrow C$ as
\begin{equation}
x_{k+1}=J_{\lambda_k}^Tx_k=\frac{1}{1+\lambda_k}x_k\oplus\frac{\lambda_k}{1+\lambda_k}T(J_{\lambda_k}^Tx_k).
\label{ppa}\end{equation}
\begin{lemma}
Let $T:C\rightarrow C$ be a quasi nonexpansive mapping, then $F(T)$ is closed and convex.
\end{lemma}
\begin{proof}
Take $p, q\in F(T)$ and $t\in [0,1]$, we show that $tp\oplus(1-t)q \in F(T)$ or equivalently $d(tp\oplus(1-t)q, T (tp\oplus(1-t)q))=0$. Note that\\
$d^2(tp\oplus(1-t)q, T (tp\oplus(1-t)q))\leq td^2(p, T (tp\oplus(1-t)q))+(1-t)d^2(q, T (tp\oplus(1-t)q))-t(1-t)d^2(p,q)\leq td^2(p, tp\oplus(1-t)q)+(1-t)d^2(q, tp\oplus(1-t)q)-t(1-t)d^2(p,q)=t(1-t)^2d^2(p,q)+t^2(1-t)d^2(p,q)-t(1-t)d^2(p,q)=0$, i.e $F(T)$ is convex.\\
Now, take $p_k\in F(T)$ such that $p_k\rightarrow p$. Note that
$d(p_k,Tp)\leq d(p_k,p)\rightarrow 0$,
i.e. $p\in F(T)$.
\end{proof}

\begin{lemma}\label{lem-sqn}
Let $T:C\rightarrow C$ be a quasi-nonexpansive mapping and $\alpha$-Lipschitz with $\alpha>1$. If $\{\lambda_k\}$ is a positive sequence, then $J_{\lambda_k}^T$ is a strongly quasi-nonexpansive sequence.
\end{lemma}
\begin{proof}
Take $\{x_k\}$ in $C$ and $p\in F(T)$. Now, by definition of $J_{\lambda_k}^T$, we have\\
$d^2(J_{\lambda_k}^Tx_k, p)=d^2(\frac{1}{1+\lambda_k}x_k\oplus\frac{\lambda_k}{1+\lambda_k}T(J_{\lambda_k}^Tx_k), p)\leq
\frac{1}{1+\lambda_k}d^2(x_k,p)+\frac{\lambda_k}{1+\lambda_k}d^2(T(J_{\lambda_k}^Tx_k), p)-\frac{\lambda_k}{(1+\lambda_k)^2}d^2(x_k,T(J_{\lambda_k}^Tx_k))\leq \frac{1}{1+\lambda_k}d^2(x_k,p)+\frac{\lambda_k}{1+\lambda_k}d^2(J_{\lambda_k}^Tx_k,p)
-\frac{1}{\lambda_k}d^2(x_k,J_{\lambda_k}^Tx_k)$.\\
Therefore
$$d^2(x_k,J_{\lambda_k}^Tx_k)\leq \frac{\lambda_k}{1+\lambda_k}(d^2(x_k,p)-d^2(J_{\lambda_k}^Tx_k,p))$$
which shows $J_{\lambda_k}^T$ is strongly quasi nonexpansive.
\end{proof}

\begin{lemma}\label{lem-con}
Let $T:C\rightarrow C$ be an $\alpha$-Lipschitz with $\alpha>1$, demiclosed and quasi-nonexpansive mapping. If $\{\lambda_k\}$ is a positive sequence such that $\liminf \lambda_k>0$, then $J_{\lambda_k}^T$ satisfies \eqref{condition}.
\end{lemma}
\begin{proof}
Let $\{x_k\}$ be an arbitrary sequence such that $x_k\overset{\triangle}{\longrightarrow} p$ and $d(x_k, J_{\lambda_k}^T x_k)\rightarrow 0$. We want to prove that $p\in \bigcap_k F(J_{\lambda_k}^T)$. Note that $d(x_k, J_{\lambda_k}^T x_k)\rightarrow 0$ implies that $\frac{\lambda_k}{1+\lambda_k}d(x_k, T(J_{\lambda_k}^T x_k))\rightarrow 0$. Since $\liminf \lambda_k>0$ hence $d(x_k,T(J_{\lambda_k}^T x_k))\rightarrow 0$. Therefore we have $d(J_{\lambda_k}^T x_k, T(J_{\lambda_k}^T x_k))\rightarrow 0$. Now, since $T$ is demiclosed and $J_{\lambda_k}^T x_k\overset{\triangle}{\longrightarrow} p$, we get $p\in \bigcap_k F(T_k)$.
\end{proof}
\begin{corollary}
Let $T:C\rightarrow C$ be an $\alpha$-Lipschitz with $\alpha>1$, demiclosed, quasi-nonexpansive mapping and $\{\lambda_k\}$ be a positive sequence such that $\liminf \lambda_k>0$. If we define $x_{k+1}=J_{\lambda_k}^Tx_k$ such that $x_0\in C$, then the sequence $\{x_k\}$ $\Delta$-converges to an element of $F(T)$.
\end{corollary}
\begin{proof}
A consequence of Lemmas \ref{lem-sqn}, \ref{lem-con} and Theorem \ref{TTT}.
\end{proof}
\begin{theorem}
Let $T:C\rightarrow C$ be an $\alpha$-Lipschitz with $\alpha>1$, demiclosed  and quasi-nonexpansive. If $\liminf \lambda_k>0$ and the sequence $\{x_k\}$ generated by
$$x_{k+1}=\alpha_k u\oplus(1-\alpha_k)J_{\lambda_k}^T x_k,$$
where $u, x_1\in C$ and the sequence $\{\alpha_k\}\subset(0,1)$ satisfies $\lim \alpha_k=0$ and $\sum_{k=1}^{+\infty}\alpha_k= +\infty$, then $\{x_k\}$ converges strongly to  ${\rm Proj}_{F(T)}u$.
\end{theorem}
\begin{proof}
$J_{\lambda_k}^T$ is strongly quasi-nonexpansive sequence by Lemma \ref{lem-sqn}. Also Lemma \ref{lem-con} shows that the sequence $\{J_{\lambda_k}^T\}$
satisfies \eqref{condition}. Now, Theorem \ref{theo-hal} implies that $\{x_k\}$ converges strongly to  ${\rm Proj}_{\bigcap_kF(J_{\lambda_k}^T)}u$.
The result follows because $F(J_{\lambda_k}^T)=F(T)$.
\end{proof}

\subsection{\bf Pseudo-monotone Equilibrium Problems}
Let $K\subseteq X$ be closed and convex. Suppose that $f:K\times K\rightarrow\mathbb{R}$ is a bifunction. we recall the definitions of pseudo-monotone and $\theta$-under monotone bifunctions.\\ $f$ is called pseudo-monotone, iff \\ Whenever $f(x,y)\geq0$ with $x,y\in K$ it holds that $f(y,x)\leq0$.\\
$f$ is called $\theta$-under monotone, iff\\
 There exists $\theta\geq0$ such that
$f(x,y)+f(y,x)\leq \theta d^2(x,y)$, for all $x,y\in K$.\\ In \cite{km-1} has been shown that for a given $x\in X$ and $\lambda>\theta$, there is a unique point denoted by $J_{\lambda}^fx$ such that
\begin{equation}
f(J_{\lambda}^fx,y)+\lambda\langle
\overrightarrow{ xJ_{\lambda}^fx}, \overrightarrow{J_{\lambda}^fxy}\rangle\geq0, \ \ \ \ \forall y\in K
\label{EP-resolvent}\end{equation}
$J_{\lambda}^fx$ is called the resolvent of $f$ of order $\lambda$ at $x \in X$. Take a sequence of regularization parameters $\{\lambda_k\}\subset
(\theta, \bar{\lambda}]$,
for some $\bar{\lambda} > \theta$ and $x_0 \in X$. The proximal point algorithm for approximation of an equilibrium point of $f$ proposed by $x_{k+1}=J_{\lambda_k}^fx_k$that studied by Iusem and Sosa in \cite{is-2} in Hilbert spaces. In this subsection we show that $\Delta$-  convergence of the proximal point algorithm and its Halpern version to an equilibrium point of $f$ is a consequence of the results of  Section 2 by assuming existence
of a sequence that satisfies \eqref{EP-resolvent}. The set of all equilibrium point of $f$ is denoted by $S(f,K)$.
In \eqref{EP-resolvent}, it is obvious that $F(J_{\lambda}^f)\subseteq S(f,K)$ and if $f$ is pseudo-monotone, then  $S(f,K)\subseteq F(J_{\lambda}^f)$.

\begin{lemma}\label{EP-lem1}
Let $f:K\times K\rightarrow\mathbb{R}$ be a pseudo-monotone and $\theta$-under monotone  bifunction and suppose that  $f(x,x)=0$ and $f(x,\cdot)$ is lsc and convex for all $x\in K$. If $S(f,K)\neq\varnothing$ and $f(\cdot,y)$ is $\bigtriangleup$-upper semicontinuous for all $y\in K$, then $J_{\lambda_k}^f$ is strongly quasi-nonexpansive sequence.
\end{lemma}
\begin{proof}
Take $p\in S(f,K)$ and set $y=p$ in \eqref{EP-resolvent}, we obtain
$$f(J_{\lambda_k}^fx_k,p)+\lambda_k\langle\overrightarrow{ x_kJ_{\lambda_k}^fx_k}, \overrightarrow{J_{\lambda_k}^fx_kp}\rangle\geq0$$
Since $p$ is an equilibrium point and $f$ is pseudo-monotone, therefore $f(J_{\lambda_k}^fx_k,p)\leq0$. Hence
$$\langle\overrightarrow{ x_kJ_{\lambda_k}^fx_k}, \overrightarrow{J_{\lambda_k}^fx_kp}\rangle\geq0$$
which implies that $$d^2(x_k,J_{\lambda_k}^Tx_k)\leq d^2(x_k,p)-d^2(J_{\lambda_k}^Tx_k,p)$$
Therefore $J_{\lambda_k}^T$ is strongly quasi-nonexpansive.
\end{proof}
\
\begin{lemma}\label{EP-lem2}
Let $f:K\times K\rightarrow\mathbb{R}$ be a pseudo-monotone and $\theta$-under monotone  bifunction and suppose that  $f(x,x)=0$ and $f(x,\cdot)$ is lsc and convex for all $x\in K$. If $S(f,K)\neq\varnothing$ and $f(\cdot,y)$ is $\bigtriangleup$-upper semicontinuous for all $y\in K$, then $J_{\lambda_k}^f$ satisfies \eqref{condition}.
\end{lemma}
\begin{proof}
Fix $y\in K$. Let $\{x_k\}$ be an arbitrary sequence such that $x_k\overset{\triangle}{\longrightarrow} p$ and $d(x_k, J_{\lambda_k}^f x_k)\rightarrow 0$. We want to prove $p\in \bigcap_k F(J_{\lambda_k}^f)$. Note that\\
$0\leq f(J_{\lambda_k}^f x_k,y)+\lambda_k \langle  \overrightarrow{x_kJ_{\lambda_k}^f x_k },\overrightarrow{J_{\lambda_k}^f x_ky }\rangle $
$\leq f(J_{\lambda_k}^f x_k,y)+\lambda_k d(x_k,J_{\lambda_k}^f x_k)d(J_{\lambda_k}^f x_k,y)$.\\
Since $\{\lambda_k\}$ and $\{x_k\}$ are bounded and  $\lim d(J_{\lambda_k}^f x_k,x_k)=0$, we have:
\begin{equation}
0\leq \liminf f(J_{\lambda_k}^f x_k,y), \   \   \  \forall y\in K.
\label{}\end{equation}
On the other hand, since $\lim d(J_{\lambda_k}^f x_k,x_k)=0$, therefore $J_{\lambda_k}^f x_k\overset{\triangle}{\longrightarrow}p$. Now since $f(\cdot,y)$ is $\triangle$-upper semicontinuous for all $y\in K$, we have:
$$0\leq \liminf f(J_{\lambda_k}^f x_k,y)\leq \limsup f(J_{\lambda_k}^f x_k,y)\leq f(p,y)$$
for all $y\in K$. So that $p\in S(f,K)$, i.e. $p\in \bigcap_k F(J_{\lambda_k}^f)$.
\end{proof}
The following theorem is one of the consequences of Section 2 (to see an independent proof, see \cite{km-1}).
\begin{theorem}\label{delta convergence ppa}
Let $f:K\times K\rightarrow\mathbb{R}$ be a pseudo-monotone and $\theta$-under monotone  bifunction and suppose that  $f(x,x)=0$ and $f(x,\cdot)$ is lsc and convex for all $x\in K$. If $S(f,K)\neq\varnothing$ and $f(\cdot,y)$ is $\bigtriangleup$-upper semicontinuous for all $y\in K$, then the sequence $\{x_k\}$ generated by (\ref{EP-resolvent}), is $\triangle$-convergent to an element of $S(f,K)$.\\
\end{theorem}
\begin{proof}
It is a consequence of Lemma \ref{EP-lem1}, Lemma \ref{EP-lem2} and Theorem \ref{TTT}.
\end{proof}

Take a sequence of regularization parameters $\{\lambda_k\}\subset
(\theta, \bar{\lambda}]$,
for some $\bar{\lambda} > \theta$ and $x_0 \in X$.
Consider the following Halpern regularization of the proximal point algorithm for equilibrium problem:
\begin{equation}\begin{cases}f(J_{\lambda_k}^fx_k,y)+\lambda_{k}\langle \overrightarrow{x_{k}J_{\lambda_k}^fx_k},\overrightarrow{J_{\lambda_k}^fx_ky}\rangle \geq 0 ,\ \ \ \ \   \forall y\in K,\\
x_{k+1}=\alpha_k u\oplus(1-\alpha_k)J_{\lambda_k}^fx_k,\\
\end{cases}\label{bi-hc}\end{equation}
where $u \in X$ and the sequence $\{\alpha_k\}\subset(0,1)$ satisfies $\lim \alpha_k=0$ and $\sum_{k=1}^{+\infty}\alpha_k= +\infty$. We will prove the strong convergence of the generated sequence by (\ref{bi-hc}) to an equilibrium point of $f$ by assuming existence
of a sequence that satisfies \eqref{bi-hc}. In fact, we prove $x_k\rightarrow x^*=Proj_{S(f,K)}u$.

\begin{theorem}\label{strong convergence halpern}
Let $f:K\times K\rightarrow\mathbb{R}$ be a pseudo-monotone and $\theta$-under monotone  bifunction and suppose that  $f(x,x)=0$ and $f(x,\cdot)$ is lsc and convex for all $x\in K$. If $S(f,K)\neq\varnothing$ and $f(\cdot,y)$ is $\bigtriangleup$-upper semicontinuous for all $y\in K$, then $\{x_k\}$ generated by \eqref{bi-hc} converges strongly to ${\rm Proj}_{S(f,K)}u$.
\end{theorem}
\begin{proof}
A consequence of Lemmas \ref{EP-lem1}, \ref{EP-lem2} and Theorem \ref{theo-hal}.
\end{proof}

\end{document}